\documentclass{amsart}
\usepackage[all]{xy}
\usepackage{amsmath,amsfonts,amssymb,amsthm,epsfig,amscd,comment,latexsym,psfrag}

\newtheorem{Theorem}{Theorem}[section]
\newtheorem{prop}{Proposition}

\newtheorem{conj}{Conjecture}
\newtheorem{lemma}{Lemma}

\theoremstyle{definition}
\newtheorem{Remark}[Theorem]{Remark}

\def\I{{I}}

\def\hX{\widehat{X}}
\def\hY{\widehat{Y}}

\def\hg{\hat{\mathfrak{g}}}
\def\1{{{1}}}

\def\la{\langle}
\def\ra{\rangle}

\def\k{\Bbbk}

\def\sl{\mathfrak{sl}}

\def\h{\mathfrak{h}}

\def\Z{\mathbb Z}
\def\N{\mathbb N} 
\def\C{\mathbb C}
\def\g{\mathfrak{g}}
\def\h{\widehat{\mathfrak{h}}}
\def\hg{\widehat{\mathfrak{g}}}

\def\l{\lambda}

\def\Sym{\mathrm{Sym}}

\def\id{\mathrm{id}}

\newcommand{\dU}{\dot{U}}

\psfrag{Qpm}{$Q_{+-}$} \psfrag{Qmp}{$Q_{-+}$}
\psfrag{Ione}{$\quad\mathbf{1}$}
\psfrag{Xiota}{$\iota$}\psfrag{Xiotap}{$\iota'$}
\psfrag{Xgi}{$g_i$}\psfrag{Xgim1}{$g_i^{-1}$}
\psfrag{Xsumi}{\large{$\sum_{i\in I}$}}
\psfrag{Xbeta}{$\overline{\beta}$}
\psfrag{Xgj}{$g_j$} \psfrag{ifinotj}{\large{if $i\not= j$}}
\psfrag{XSnLm}{$S^n_-\otimes \Lambda^m_+$} 
\psfrag{XLmSn}{$\Lambda^m_+\otimes S^n_-$}
\psfrag{XLm1Sn1}{$\Lambda^{m-1}_+\otimes S^{n-1}_-$}
\psfrag{XSumkb}{$-\sum_{b=0}^{k-2} (k-b-1)$}

\title{Loop realizations of quantum affine algebras}

\begin{document} 
\setcounter{tocdepth}{1}

\author{Sabin Cautis}
\email{cautis@usc.edu}
\address{Department of Mathematics\\ University of Southern California \\ Los Angeles, CA}

\author{Anthony Licata}
\email{amlicata@gmail.com}
\address{Department of Mathematics \\ Australian National University \\ Canberra, Australia}

\begin{abstract}
We give a simplified description of quantum affine algebras in their loop presentation. This description is related to Drinfeld's new realization via halves of vertex operators. We also define an idempotent version of the quantum affine algebra which is suitable for categorification.
\end{abstract}

\maketitle
\tableofcontents

\section{Introduction}

Let $\g$ be a finite dimensional simple simply-laced complex Lie algebra.  There are two well-known ways to add a parameter to $\g$.  The first is $q$-deformation, where one deforms the enveloping algebra 
$U(\g)$ to a new algebra $U_q(\g)$, commonly known as a quantum group.  The second is affinization, where one replaces the Lie algebra $\g$ by the affine Lie algebra $\hg := \g \otimes \k[t,t^{-1}] \oplus \k c$.  The Lie algebra $\hg$ is a central extension of the loop algebra of $\g$.   Both affine Lie algebras and quantum groups arise in representation theory, low dimensional topology, algebraic geometry, and many other parts of mathematics and mathematical physics. 

The quantum affine algebra $U_q(\hg)$, which is a deformation of the enveloping algebra of $\hg$, combines both ideas.  However, the description of $U_q(\hg)$ is much more complicated than that of either $U_q(\g)$ or $\hg$.  Perhaps on account of connections to physics, the relations in $U_q(\hg)$ are typically encoded in generating functions (Section \ref{sec:drinfeld}), and the resulting presentation is known as Drinfeld's realization \cite{D}.  A central role in this presentation is played by the quantum Heisenberg algebra $\h \subset U_q(\hg)$, since many of the relations in $U_q(\hg)$ are expressed using generating functions whose terms are in $\h$. In particular, explicit constructions of representations of $U_q(\hg)$ (such as the Frenkel-Jing homogeneous construction of the basic representation \cite{FJ}) as well as explicit isomorphisms between different presentations of $U_q(\hg)$ (for example the work of Damiani and Beck \cite{Da,B}) often make crucial use of the quantum Heisenberg algebra $\h$.

The main motivation for the current note comes from categorification in the representation theory of quantum affine algebras.  In the accompanying paper \cite{CL2}, we construct a 2-representation of the quantum affine (and quantum toroidal) algebra on the derived categories of Hilbert schemes of points on the surface $\widehat{\C^2/\Gamma}$, where $\Gamma \subset SL_2(\C)$ is the finite subgroup associated to $\g$ under the McKay correspondence.  After passing to equivariant K-theory, this gives a representation of $U_q(\hg)$.  However, the presentation of $U_q(\hg)$ which appears naturally in this construction is not identical to Drinfeld's new realization (though it is very similar).  Thus, in the process of producing an action of $U_q(\hg)$ on equivariant K-theory, we were forced to show an equivalence between various presentations of the quantum affine algebra.   
Since these results seems to be of some independent interest and since they do not require 2-categories or algebraic geometry, we decided to write them down independently from \cite{CL2}.   We emphasize, however, that this paper, and Theorem \ref{thm:main} in particular, has application to the proof that quantum affine algebras act on the equivariant K-theory of Hilbert schemes.

We begin the current paper by rewriting the Drinfeld realization using halves of vertex operators. This gives a different set of generators of the quantum Heisenberg subalgebra $\h$ (Section \ref{sec:vertex}). Next we describe an idempotent modification $\dU_q(\hg)$ of the quantum affine algebra in Section \ref{sec:idempotent}. The idempotent version $\dU_q(\hg)$ is not a unital algebra anymore, but it has a somewhat simplified presentation. Moreover, any representation of $\dU(\hg)$ automatically gives rise to a representation of the quantum affine algebra. The presentation of $\dU_q(\hg)$ which we use in \cite{CL2} is, up to a slight renomalization, the presentation in Section \ref{sec:idempotent}. We describe this renormalized realization in Section \ref{sec:renormalized}. Much of this is known to experts in the field. 

Our main contribution is to explain how the relations in Drinfeld's new realization (and also in our idempotent modification) are far from being minimal. More precisely, in Section \ref{sec:minimal} we give a smaller set of relations in $\dU_q(\hg)$ and prove (Theorem \ref{thm:main}) that all other relations are a consequence of these relations. In the last section we speculate on a minimal realization of $\dU_q(\hg)$ (Conjecture \ref{conj:1}) which is generated only by $E$s and $F$s and contains just four kinds of relations. 

\noindent {\bf Acknowledgments:}
We would like to thank Pavel Etingof, Sachin Gautam and David Hernandez for some helpful discussions, and to the referee for several useful comments. S.C. was supported by NSF grants DMS-0964439, DMS-1101439 and the Alfred P. Sloan foundation. A.L. would like to acknowledge the support of the Institute for Advanced Study.

\section{Main definitions and results}

\subsection{Dynkin data}

We work over a base field $\k$ of characteristic zero. Fix a simply-laced Dynkin diagram of finite type and denote its vertex set by $\I$ and the associated Lie algebra by $\g$. We denote the weight lattice of $\g$ by $X$ and the root lattice of $\g$ by $Y$. Thus $Y$ is a sublattice of $X$. We equip $X$ with the standard pairing $\la \cdot, \cdot \ra$. For $i \in \I$, $\alpha_i \in Y$ and $\Lambda_i \in X$ will denote the simple roots and fundamental weights. We will often write $i$ instead of $\alpha_i$, especially in the pairing; thus for $\l \in X$ we will write $\la \l, i \ra$ instead of $\la \l, \alpha_i \ra$.  

The pairing of simple roots satisfies $\la i, j \ra= C_{i,j}$, where $C_{i,j}$ is the Cartan matrix of $\g$. In particular, we have:
\begin{itemize}
\item $\la i, i \ra = 2$ for all $i\in \I$,
\item $\la i, j \ra = -1$ when $i\neq j \in \I$ are joined by an edge,
\item $\la i,j \ra = 0$ when $i\neq j \in \I$ are not joined by an edge, and
\item $\la \Lambda_i, j \ra = \delta_{i,j}$ for all $i, j \in \I.$
\end{itemize}

Let $\hg = \g \otimes \k[t,t^{-1}] \oplus \k c$ be the affine Lie algebra associated to $\g$, and denote by $\hX$ the affine weight lattice. The affine root lattice is denoted $\hY$. Note that $\hY = Y \oplus \Z \delta$ where $\la i, \delta \ra = 0$, $\la \Lambda_i, \delta \ra = 1$ for all $i \in \I$, and $\la \delta, \delta \ra = 0$. 

\subsection{Graded vector spaces}

Consider a $\Z$-graded finite dimensional vector space $V = \oplus_i V(i)$. One can associate to $V$ the polynomial $f_V := \sum_i q^i \dim V(i)$. This gives a bijection between isomorphism classes of finite dimensional graded vector spaces and elements $f \in \N[q,q^{-1}]$. 

From $V$ one can construct the associated $\Z$-graded vector spaces $\Sym^n(V)$ and $\Lambda^n(V)$. If $V$ has graded dimension $f \in \N[q,q^{-1}]$ then we denote by $\Sym^n(f)$ and $\Lambda^n(f)$ the graded dimensions of the $\Z$-graded vector spaces $\Sym^n(V)$ and $\Lambda^n(V)$.  For example, if $f = q+q^{-1}$ then $\Sym^n(f) = q^{n} + q^{n-2} + \dots + q^{-n+2} + q^{-n}$ which is just the quantum integer $[n+1]$. On the other hand, $\Lambda^n(f)$ is $1,[2],1$ if $n=0,1,2$, and is zero otherwise.

\subsection{The Drinfeld realization}\label{sec:drinfeld}

We begin with the Drinfeld realization of the quantum affine algebra as defined in \cite[Sec. 1.2]{Nak1}. It has generators $e_{i,r}, f_{i,r}$ ($i \in \I$, $r \in \Z$), $q^h$ ($h \in X^*$) and $h_{i,m}$, $q^{\pm d}$, $q^{\pm c/2}$ where $i \in \I$ and $m \in \Z \setminus \{0\}$. The set of relations are as follows:
\begin{enumerate}
\item $q^{\pm c/2}$ is central 
\item $q^0=1$, $q^h q^{h'} = q^{h+h'}$, $[q^h, h_{i,m}]=0$, $q^d q^{-d} = 1$, $q^{c/2} q^{-c/2} = 1$
\item $\psi_i^\pm(z) \psi_j^{\pm}(w) = \psi_j^{\pm}(w) \psi_i^{\pm}(z)$
\item $\psi^-_i(z) \psi_j^+(w) = \frac{(z-q^{- \la i, j \ra} q^c w)(z - q^{\la i, j \ra} q^{-c} w)}{(z-q^{\la i, j \ra} q^c w)(z - q^{- \la i, j \ra} q^{-c} w)} \psi^+_j(w) \psi_i^-(z)$
\item $[q^d, q^h]=0$, $q^d h_{i,m} q^{-d} = q^m h_{i,m}$ and $q^d e_{i,r} q^{-d} = q^r e_{i,r}$, $q^d f_{i,r} q^{-d} = q^r f_{i,r}$
\item $(q^{\pm s c/2} z - q^{\pm \la i, j \ra} w) \psi_j^s(z) x_i^\pm(w) = (q^{\pm \la i, j \ra} q^{\pm sc/2} z - w) x_i^{\pm}(w) \psi_j^s(z)$ where $s = \pm$
\item $[x_i^+(z),x_j^-(w)] = \delta_{ij} \frac{1}{q-q^{-1}} \big(\delta(q^c wz^{-1}) \psi_i^+(q^{\frac{c}{2}} w) - \delta(q^c zw^{-1}) \psi^-_i(q^{\frac{c}{2}}z) \big)$
\item $(z-q^{\pm 2}w) x_i^\pm(z) x_i^\pm(w) = (q^{\pm 2} z -w) x_i^\pm(w) x_i^\pm(z)$
\item $(z-q^{\mp 1} w) x_i^\pm(z) x_j^\pm(w) = x_j^\pm(w) x_i^\pm(z) (q^{\mp 1} z-w)$ if $\la i, j \ra = -1$
\item if $\la i,j \ra \le 0$ then
$$\sum_{\sigma \in S_N} \sum_{s=0}^{N = 1 - \la i,j \ra} (-1)^s \left[ \begin{matrix} N \\ s \end{matrix} \right] x_i^\pm(z_{\sigma(1)}) \hdots x_i^\pm(z_{\sigma(s)}) x_j^\pm(w) x_i^\pm(z_{\sigma(s+1)}) \hdots x_i^\pm(z_{\sigma(N)}) = 0.$$
\end{enumerate}
In the above relations we have
\begin{eqnarray*}
x_i^+(z) &=& \sum_{n \in \Z} e_{i,n} z^{-n} \hspace{1cm} 
x_i^-(z) = \sum_{n \in \Z} f_{i,n} z^{-n} \\
\psi_i^\pm(z) &=& \sum_{n \geq 0} \psi_i^\pm(\pm n) z^{\mp n} = q^{\pm h_i} \exp \big(\pm (q-q^{-1}) \sum_{n>0} h_{i, \pm n} z^{\mp n} \big) \\
\delta(z) &=& \sum_{n \in \Z} z^n
\end{eqnarray*}
Notice that we only deal with the case when the affine Lie algebra is simply laced. 

\subsection{The vertex realization}\label{sec:vertex}

We now rewrite Drinfeld's realization in terms of halves of vertex operators. This means that instead of $h$'s we will use $P$s and $Q$s defined as homogeneous components in $z$ of generating functions:
\begin{eqnarray*}
&& \sum_{n \ge 0} Q_i^{(n)} z^n := \exp \left( \sum_{n \ge 1} \frac{h_{i,n}}{[n]} z^n \right) \text{ and } \sum_{n \ge 0} (-1)^n Q_i^{(1^n)} z^n := \exp \left( - \sum_{n \ge 1} \frac{h_{i,n}}{[n]} z^n \right) \\ 
&& \sum_{n \ge 0} P_i^{(n)} z^n := \exp \left( \sum_{n \ge 1} \frac{h_{i,-n}}{[n]} z^n \right) \text{ and } \sum_{n \ge 0} (-1)^n P_i^{(1^n)} z^n := \exp \left( - \sum_{n \ge 1} \frac{h_{i,-n}}{[n]} z^n \right).
\end{eqnarray*}
These operators were also considered in \cite{FJ} and \cite[Lemma 3.2]{CP2}. We also define
\begin{equation}\label{eq:rescale}
E_{i,r} := (q^{-c/2})^r e_{i,r} \text{ and } F_{i,r} := (q^{-c/2})^r f_{i,r}.
\end{equation}
For convenience we also define 
\begin{eqnarray*}
Q_i^{[1^n]} := \sum_{m=0}^n (-q)^m [m] Q_i^{(1^{n-m})} Q_i^{(m)} &\text{ and }&
Q_i^{[n]} := \sum_{m=0}^n (-q)^m [m] Q_i^{(n-m)} Q_i^{(1^m)} \\
P_i^{[1^n]} := \sum_{m=0}^n (-q)^{-m} [m] P_i^{(1^{n-m})} P_i^{(m)} &\text{ and }&
P_i^{[n]} := \sum_{m=0}^n (-q)^{-m} [m] P_i^{(n-m)} P_i^{(1^m)} 
\end{eqnarray*}
Note that $P_i^{[1]} = -q^{-1}P_i$ and $Q_i^{[1]} = -qQ_i$.

Thus, we have generators $E_{i,r},F_{i,r}$, $q^h$ ($h \in X^*$), $q^{\pm d}$, $q^{\pm c/2}$,$P_i^{(n)}, P_i^{(1^n)}$ $Q_i^{(n)}$ and $Q_i^{(1^n)}$ where $i \in \I, r \in \Z$ and $n \in \N$. The set of relations between these generators are now given by
\begin{enumerate}
\item $q^{\pm c/2}$ is central 
\item $q^0=1$, $q^h q^{h'} = q^{h+h'}$, $[q^h, P^{(n)}] = 0 = [q^h, Q^{(n)}]$, $q^d q^{-d} = 1$, $q^{c/2} q^{-c/2} = 1$
\item The generators $\{P_i^{(n)},P_i^{(1^n)}\}_{i \in \I}$ commute among each other and likewise $\{Q_i^{(n)},Q_i^{(1^n)}\}_{i \in \I}$ commute among each other.
\item We have 
\begin{eqnarray}
\label{eq:PQ1} Q_j^{(n)} P_i^{(m)} &=&
\begin{cases}
\sum_{k \ge 0} \Sym^k([2][c]) P_i^{(m-k)} Q_i^{(n-k)} \text{ if } i = j \\
\sum_{k \ge 0} (-1)^k \Lambda^k([c])P_i^{(m-k)} Q_j^{(n-k)} \text{ if } \la i, j \ra = -1 \\
P_i^{(m)} Q_j^{(n)} \text{ if } \la i, j \ra = 0
\end{cases} \\
\label{eq:PQ2}
Q_j^{(1^n)} P_i^{(m)} &=&
\begin{cases}
\sum_{k \ge 0} \Lambda^k([2][c]) P_i^{(m-k)} Q_i^{(1^{n-k})} \text{ if } i = j \\
\sum_{k \ge 0} (-1)^k \Sym^k([c])  P_i^{(m-k)} Q_j^{(1^{n-k})} \text{ if } \la i, j \ra = -1 \\
P_i^{(m)} Q_j^{(1^n)} \text{ if } \la i, j \ra = 0
\end{cases}
\end{eqnarray}
and likewise if you exchange $(a)$ and $(1^a)$ everywhere.

\item $[q^d, q^h]=0$, $q^d Q_i^{(m)} q^{-d} = q^m Q_i^{(m)}$ $q^d P_i^{(m)} q^{-d} = q^{-m} P_i^{(m)}$ and $q^d (E_{i,r}) q^{-d} = q^r (E_{i,r})$, $q^d (F_{i,r}) q^{-d} = q^r (F_{i,r})$

\item We have
\begin{eqnarray*}
\label{eq:qE1}
q^c [Q_i^{[1^{a+1}]}, E_{i,b}] &=&
\begin{cases}
q^2 Q_i^{[1^a]} E_{i,b+1} - q^{-2} E_{i,b+1} Q_i^{[1^a]} \text{ if } a > 0 \\
[2]E_{i,b+1} \text{ if } a = 0.
\end{cases} \\
\label{eq:qE2}
[Q_i^{[1^{a+1}]}, F_{i,b}] &=&
\begin{cases}
q^{-2} Q_i^{[1^a]} F_{i,b+1} - q^2 F_{i,b+1} Q_i^{[1^a]} \text{ if } a > 0 \\
- [2]F_{i,b+1}\text{ if } a = 0.
\end{cases} \\
\label{eq:qE3}
[P_i^{[1^{a+1}]}, E_{i,b+1}] &=&
\begin{cases}
q^2 E_{i,b} P_i^{[1^{a}]} - q^{-2} P_i^{[1^{a}]} E_{i,b} \text{ if } a > 0 \\
[2] E_{i,b} \text{ if } a = 0
\end{cases} \\
\label{eq:qE4}
q^{-c} [P_i^{[1^{a+1}]}, F_{i,b+1}] &=&
\begin{cases}
q^{-2} F_{i,b} P_i^{[1^{a}]} - q^2 P_i^{[1^{a}]} F_{i,b} \text{ if } a > 0 \\
- [2] F_{i,b} \text{ if } a = 0.
\end{cases}
\end{eqnarray*}
while if $\la i, j \ra = -1$ we have
\begin{eqnarray*}
\label{eq:qE5}
q^c [Q_j^{[1^{a+1}]}, E_{i,b}] &=&
\begin{cases}
- qE_{i,b+1} Q_j^{[1^a]} + q^{-1} Q_j^{[1^a]} E_{i,b+1}\text{ if } a > 0 \\
- E_{i,b+1} \text{ if } a = 0.
\end{cases} \\
\label{eq:qE6}
[Q_j^{[1^{a+1}]}, F_{i,b}] &=&
\begin{cases}
- q^{-1} F_{i,b+1} Q_j^{[1^a]} + q Q_j^{[1^a]} F_{i,b+1} \text{ if } a > 0 \\
F_{i,b+1} \text{ if } a = 0
\end{cases} \\
\label{eq:qE7}
[P_j^{[1^{a+1}]}, E_{i,b+1}] &=&
\begin{cases}
- q^{-1}E_{i,b} P_j^{[1^{a}]} + q P_j^{[1^{a}]} E_{i,b} \text{ if } a > 0 \\
- E_{i,b} \text{ if } a = 0
\end{cases} \\
\label{eq:qE8}
q^{-c} [P_j^{[1^{a+1}]}, F_{i,b+1}] &=&
\begin{cases}
- qF_{i,b} P_j^{[1^{a}]} + q^{-1} P_j^{[1^{a}]} F_{i,b} \text{ if } a > 0 \\
F_{i,b} \text{ if } a = 0;
\end{cases}
\end{eqnarray*}
if $\la i,j \ra = 0$ we have that $P_j^{[1^a]}$ and $Q_j^{[1^a]}$ both commute with $E_{i,b}$ and $F_{i,b}$.

\item We have
$$[E_{i,a}, F_{i,b}] =
\begin{cases}
q^{ac} q^{h_i} Q_i^{[1^{a+b}]} \text{ if } a+b > 0 \\
q^{bc} q^{-h_i} P_i^{[1^{-a-b}]} \text{ if } a+b < 0 \\
\frac{q^{ac} q^{h_i}  - q^{-ac} q^{-h_i}}{q-q^{-1}} \text{ if } a+b=0.
\end{cases}$$
while if $i \ne j$ then $[E_{i,a}, F_{j,b}] = 0$.
\item For any $m,n \in \Z$ we have
\begin{eqnarray*}
E_{i,m}E_{i,n-1} + E_{i,n}E_{i,m-1} &=& q^2 \left( E_{i,m-1}E_{i,n} + E_{i,n-1}E_{i,m}\right) \\
F_{i,n-1}F_{i,m} + F_{i,m-1}F_{i,n} &=& q^2 \left( F_{i,n} F_{i,m-1} + F_{i,m} F_{i,n-1} \right).
\end{eqnarray*}

\item For any $m,n \in \Z$, if $\la i, j \ra = -1$ we have
\begin{eqnarray*}
E_{i,m}E_{j,n-1} + E_{j,n}E_{i,m-1} &=& q^{-1} \left( E_{j,n-1}E_{i,m} + E_{i,m-1}E_{j,n} \right) \\
F_{i,m-1}F_{j,n} + F_{j,n-1}F_{i,m-1} &=& q^{-1} \left( F_{j,n} F_{i,m-1} + F_{i,m} F_{j,n-1} \right)
\end{eqnarray*}
while if $\la i, j \ra = 0$ then
$$E_{i,m}E_{j,n} = E_{j,n}E_{i,m} \text{ and } F_{i,m}F_{j,n} = F_{j,n}F_{i,m}.$$

\item If $\la i, j \ra = -1$ then
$$\sum_{\sigma \in S_2} \left( E_{j,n}E_{i,m_{\sigma(1)}}E_{i,m_{\sigma(2)}} +  E_{i, m_{\sigma(1)}}E_{i,m_{\sigma(2)}}E_{j,n} \right) = \sum_{\sigma \in S_2} [2]E_{i,m_{\sigma(1)}}E_{j,n}E_{i,m_{\sigma(2)}}$$
and similarly if we replace all $E$s by $F$s.
\end{enumerate}
\begin{prop}\label{prop:main}
The Drinfeld and vertex realizations of quantum affine algebras are equivalent. 
\end{prop}
\begin{proof}
All but one of the relations in the vertex realization are obtained directly from the Drinfeld realization by writing out the condition. The only exception is condition (4) involving the commutation of $P$s and $Q$s. The fact that it is equivalent to condition (4) in the Drinfeld realization was checked in \cite{CL1} when $c=1$ (i.e. in the level one case). The same proof extends without complications to an arbitrary $c \in \N$ using the relation 
$$[h_{i,m}, h_{j,n}] = \delta_{m,-n} [\la i,j \ra n] \frac{[nc]}{n}.$$
\end{proof}

\subsection{The idempotent realization}\label{sec:idempotent}

Any representation $V= \oplus_{\l\in \hX} V(\l)$ of $U_q(\hg)$ with a weight space decomposition has a natural collection of idempotent endomorphisms, namely, for each $\l \in \hX$ there is the endomorphism given by projection onto the weight space $V(\l)$. It is therefore natural to consider an idempotent modification $\dU_q(\hg)$ of the quantum affine algebra, with the unit replaced by this collection of idempotent endomorphisms, one for each weight $\l \in \hX$.  Then giving a representation of $\dU_q(\hg)$ will be equivalent to giving a representation of $U_q(\hg)$ together with a weight space decomposition.  This point of view is used frequently in the literature on Kac-Moody categorification, since, so far at least, all categorified representations have a weight space decomposition.

For any $\l \in \hX$ denote by $\1_\l$ the idempotent which projects onto this weight space $\l$. We also fix $c$ to be a positive integer.  We define the algebra $\dU_q(\hg)$ via generators and relations as follows.

The generators are
$$E_{i,r} \1_\l, F_{i,r} \1_\l, Q_i^{(n)} \1_\l, P_i^{(n)} \1_\l, Q_i^{(1^n)} \1_\l, P_i^{(1^n)} \1_\l, \text{ where } i \in \I, \l \in \hX \text{ and } r,n \in \Z.$$
Note that we no longer have generators $q^h$, $q^{\pm d}$ or $q^{\pm c/2}$.
\begin{enumerate}
\item This condition is redundant
\item $\{\1_\l: \l \in \hX \}$ are mutually orthogonal idempotents, moreover
\begin{eqnarray*}
& & E_{i,r} \1_\l = \1_\mu E_{i,r} \1_\l = \1_\mu E_{i,r} \\
& & F_{i,-r} \1_\mu = \1_\l F_{i,-r} \1_\mu = \1_\l F_{i,-r}
\end{eqnarray*}
where $\mu = \l+ \alpha_i + rc \delta$.
\item Same as the corresponding vertex realization relation.
\item Same as the corresponding vertex realization relations (equations (\ref{eq:PQ1}) and (\ref{eq:PQ2})).
\item We have
\begin{eqnarray*}
P_i^{(n)} \1_\l = \1_\mu P_i^{(n)} \1_\l = \1_\mu P_i^{(n)} &\text{ and }& P_i^{(1^n)} \1_\l = \1_\mu P_i^{(1^n)} \1_\l = \1_\mu P_i^{(1^n)} \\
Q_i^{(n)} \1_\mu = \1_\l Q_i^{(n)} \1_\mu = \1_\l Q_i^{(n)} &\text{ and }& Q_i^{(1^n)} \1_\mu = \1_\l Q_i^{(1^n)} \1_\mu = \1_\l Q_i^{(1^n)}
\end{eqnarray*}
where $\mu = \l + nc \delta$.
\item Same as the corresponding vertex realization relation.
\item We have
$$[E_{i,a}, F_{i,b}] \1_\l =
\begin{cases}
q^{ac} q^{\la \l, i \ra} Q_i^{[1^{a+b}]} \1_\l \text{ if } a+b > 0 \\
q^{bc} q^{- \la \l, i \ra} P_i^{[1^{-a-b}]} \1_\l \text{ if } a+b < 0 \\
[\la \l, i \ra + ac] \1_\l \text{ if } a+b=0
\end{cases}$$
while if $i \ne j$ then $[E_{i,a}, F_{j,b}] \1_\l = 0$.
\item Same as the corresponding vertex realization relation.
\item Same as the corresponding vertex realization relation.
\item Same as the corresponding vertex realization relation.
\end{enumerate}

Now suppose that $V$ is a representation of $\hg$ with weight space decomposition $V = \oplus_{\l \in \hX} V(\l)$. We say that $V$ is an {\bf integrable representation} if for any $\l \in \hX$ and root $\alpha \in \hY$ the weight space $V(\l+n \alpha)$ is zero for $n \gg 0$ and $n \ll 0$. In an integrable highest weight representation of $\hg$, the central element $c$ will act as $\la \l, \delta \ra \id$, where $\l$ is the highest weight.  The integer $\la \l, \delta \ra$ is called the {\bf level} of the representation. 

\subsection{Redundancy in relations}\label{sec:minimal}

Many of the relations in the presentation of $\dU_q(\hg)$ above turn out to be redundant. We now summarize a smaller set of relations.

The generators are the same as those of the previous section. However, it suffices to consider the following smaller set of relations.
\begin{enumerate}
\item $\{\1_\l: \l \in \hX \}$ are mutually orthogonal idempotents with
\begin{eqnarray*}
& & E_{i,r} \1_\l = \1_\mu E_{i,r} \1_\l = \1_\mu E_{i,r} \\
& & F_{i,-r} \1_\mu = \1_\l F_{i,-r} \1_\mu = \1_\l F_{i,-r}
\end{eqnarray*}
where $\mu = \l+ \alpha_i + rc \delta$
\item The $P$s and $Q$s satisfy the same relations as before (conditions (4) and (5) above). 
\item We have
$$[E_{i,a}, F_{i,b}] \1_\l =
\begin{cases}
q^{ac} q^{\la \l, i \ra} Q_i^{[1^{a+b}]} \1_\l \text{ if } a+b > 0 \\
q^{bc} q^{- \la \l, i \ra} P_i^{[1^{-a-b}]} \1_\l \text{ if } a+b < 0 \\
[\la \l, i \ra + ac] \1_\l \text{ if } a+b=0
\end{cases}$$
while if $i \ne j$ then $[E_{i,a}, F_{j,b}] \1_\l = 0$.
\item For any $m,n \in \Z$ we have $E_{i,n}E_{i,n-1}\1_\l = q^2 E_{i,n-1}E_{i,n}\1_\l$ and $F_{i,n-1}F_{i,n}\1_\l = q^2 F_{i,n}F_{i,n-1}\1_\l$.
\item For any $m,n \in \Z$ we have
\begin{eqnarray*}
E_{i,1}E_j\1_\l + E_{j,1}E_i\1_\l = q^{-1} \left( E_jE_{i,1}\1_\l + E_iE_{j,1}\1_\l \right) &\text{ if }& \la i, j \ra = -1 \\
E_{i,m}E_{j,n}\1_\l = E_{j,n}E_{i,m}\1_\l &\text{ if }& \la i, j \ra = 0
\end{eqnarray*}
and similarly
\begin{eqnarray*}
F_{i,-1}F_j \1_\l + F_{j,-1}F_i \1_\l = q^{-1} \left( F_jF_{i,-1} \1_\l + F_iF_{j,-1}\1_\l \right) &\text{ if }& \la i, j \ra
 = -1 \\
F_{i,m}F_{j,n} \1_\l = F_{j,n}F_{i,m} \1_\l &\text{ if }& \la i, j \ra = 0.
\end{eqnarray*}
\item \label{co:serre} If $\la i, j \ra = -1$ then
$$(E_{j,n})(E_{i,m})^2 \1_\l + (E_{i,m})^2 (E_{j,n}) \1_\l = [2](E_{i,m})(E_{j,n})(E_{i,m}) \1_\l.$$
and similarly if we replace all $E$s by $F$s.
\end{enumerate}

\begin{Theorem}\label{thm:main} 
The relations above imply all of the relations in Drinfeld's realization. Moreover, when acting on an integrable representation, condition (\ref{co:serre}) is not necessary, since it follows formally from the other relations. 
\end{Theorem}

\section{Proof of theorem \ref{thm:main}}

We need to show that the relatons in $\dU_q(\hg)$ follow from the relations in Section \ref{sec:minimal}. Conditions (1), (2), (3) and (5) are easy to check. We will verify the rest of the relations. For simplicity we will often omit the projectors $\1_\l$ and write $q^{h_i}$ instead of $q^{\la \l, i \ra} \1_\l$. 

\subsection{Proof of (7)}
If $i \ne j$ then this condition states that $[x_i^+(z), x_j^-(w)] = 0$ which means that $E_{i,a}$ and $F_{j,b}$ commute (as claimed).

We now deal with the case $i=j$. To simplify notation we drop the $i$ subscripts everywhere. So $h_{i,n}$ becomes $h_n$, $\psi^\pm_i(z)$ is just $\psi^\pm(z)$ and so on. Condition (5) then becomes
\begin{equation}\label{eq:1}
[x^+(z), x^-(w)] = \frac{1}{q-q^{-1}} \left( \delta(q^c wz^{-1}) \psi^+(q^{\frac {c}{2}}w) - \delta(q^c zw^{-1}) \psi^-(q^{\frac{c}{2}} z) \right)
\end{equation}
where
\begin{eqnarray*}
\psi^+(z) &=& \sum_{n \ge 0} \psi^+(n) z^{-n} = q^h \exp \left( (q-q^{-1}) \sum_{n \ge 1} h_n z^{-n} \right) \\
\psi^-(z) &=& \sum_{n \ge 0} \psi^-(-n) z^n = q^{-h} \exp \left( -(q-q^{-1}) \sum_{n \ge 1} h_{-n} z^n \right).
\end{eqnarray*}
Now, the coefficient of $z^{-a}w^{-b}$ of left side of equation (\ref{eq:1}) is $(-q)^{-a-b} q^{-(a+b)c/2}[E_a, F_b]$ (recall that we have rescaled using equation (\ref{eq:rescale})). On the other hand, the coefficient of $z^{-a}w^{-b}$ on the right side of (\ref{eq:1}) is 
$$\frac{1}{q-q^{-1}} \left( q^{ac} (q^{\frac{c}{2}})^{-a-b} \psi^+(a+b) - q^{bc} (q^{\frac{c}{2}})^{-a-b} \psi^-(a+b) \right).$$
Notice that if $a+b < 0$ then the first term above vanishes while if $a+b > 0$ then the second term vanishes.

{\bf Case $a+b=0$.} As expected, we end up with
$$[E_a, F_{-a}] = \frac{1}{q-q^{-1}} \left( q^{ac} q^h  - q^{-ac} q^{-h} \right)
.$$
{\bf Case $a+b>0$.} Expanding out everything we get
$$[E_a, F_b] = \frac{(-q)^{a+b}}{q-q^{-1}} q^{ac} \psi^+(a+b).$$
Now consider
\begin{eqnarray*}
\phi^+(z) &:=& \sum_{n \ge 0} Q^{(n)} z^n = \exp \left( \sum_{n \ge 1} \frac{h_n}{[n]} z^n \right) \\
\phi^-(z) &:=& \sum_{n \ge 0} (-1)^n Q^{(1^n)} z^n = \exp \left( - \sum_{n \ge 1} \frac{h_n}{[n]} z^n \right)
\end{eqnarray*}
First notice that
$$q^h \phi^+(qz^{-1}) \phi^-(q^{-1}z^{-1}) = q^h \exp \left( \sum_{n \ge 1} \frac{h_n}{[n]} (q^n - q^{-n}) z^{-n} \right) = \psi^+(z).$$
On the other hand, we also have
\begin{eqnarray*}
\phi^+(qz^{-1}) - \phi^+(q^{-1}z^{-1})
&=& \sum_{n \ge 0} (q^n - q^{-n}) Q^{(n)} z^{-n} \\
&=& (q-q^{-1}) \sum_{n \ge 0} [n] Q^{(n)} z^{-n}.
\end{eqnarray*}
If we multiply this by $\phi^-(q^{-1} z^{-1})$ and use that $\phi^+ \phi^- = 1$ we get
\begin{eqnarray*}
\phi^+(qz^{-1}) \phi^-(q^{-1} z^{-1}) - 1
&=& (q-q^{-1}) \sum_{n \ge 0} \sum_{m=0}^n [m]Q^{(m)} (-q)^{m-n} Q^{(1^{n-m})} z^{-n} \\
&=& (q-q^{-1}) \sum_{n \ge 0} (-q)^{-n} Q^{[1^n]} z^{-n}
\end{eqnarray*}
which means that if $a+b > 0$ we have
\begin{equation}\label{eq:psi+}
\frac{(-q)^{a+b}}{q-q^{-1}} \psi^+(a+b) = q^h Q^{[1^{a+b}]}.
\end{equation}
Thus we get that
$$[E_a, F_b] = q^{ac} q^h Q^{[1^{a+b}]}.$$

{\bf Case $a+b<0$.} Expanding out we get
$$[E_a, F_b] = - \frac{(-q)^{a+b}}{q-q^{-1}} q^{-ac} \psi^-(a+b).$$
Now consider
\begin{eqnarray*}
\phi^+(z) &:=& \sum_{n \ge 0} P^{(n)} z^n = \exp \left( \sum_{n \ge 1} \frac{h_{-n}}{[n]} z^n \right) \\
\phi^-(z) &:=& \sum_{n \ge 0} (-1)^n P^{(1^n)} z^n = \exp \left(-\sum_{n \ge 1} \frac{h_{-n}}{[n]} z^n \right)\end{eqnarray*}
First we have
$$q^{-h} \phi^+(q^{-1}z) \phi^-(qz) = q^{-h} \exp \left( \sum_{n \ge 1} \frac{h_{-n}}{[n]} (q^{-n} - q^{n}) z^n \right) = \psi^-(z).$$
On the other hand, we also have
\begin{eqnarray*}
\phi^+(q^{-1}z) - \phi^+(qz)
&=& \sum_{n \ge 0} (q^{-n} - q^{n}) P^{(n)} z^n \\
&=& - (q-q^{-1}) \sum_{n \ge 0} [n] P^{(n)} z^n.
\end{eqnarray*}
If we multiply this by $\phi^-(qz)$ and use that $\phi^+ \phi^- = 1$ we get
\begin{eqnarray*}
\phi^+(q^{-1}z) \phi^-(qz) - 1
&=& - (q-q^{-1}) \sum_{n \ge 0} \sum_{m=0}^n [m] P^{(m)} (-q)^{n-m} P^{(1^{n-m})} z^{n} \\
&=& - (q-q^{-1}) \sum_{n \ge 0} (-q)^n P^{[1^n]} z^{n}
\end{eqnarray*}
which means that if $a+b < 0$ we have
\begin{equation}\label{eq:psi-}
- \frac{(-q)^{a+b}}{q-q^{-1}} \psi^-(a+b) = q^{-h} P^{[1^{-a-b}]}.
\end{equation}
Thus we get that
$$[E_a, F_b] = q^{bc} q^{-h} P^{[1^{-a-b}]}.$$

\subsection{Proof of (6)} 

If we write out relation (6) from section \ref{sec:drinfeld} when $s=+$ and $i=j$ we get two relations. In the first relation (where we take $\psi_i^+$) consider the coefficient of $z^{-a}w^{-b}$ to obtain
$$q^{c/2} \psi_i^+(a+1) e_{i,b} - q^2 \psi_i^+(a) e_{i,b+1} = q^2 q^{c/2} e_{i,b} \psi_i^+(a) - e_{i,b+1} \psi_i^+(a).$$
We show in the proof of relation (7) that if $\ell > 0$ then
$$\psi_i^+(\ell) = (q-q^{-1}) q^h Q_i^{[1^\ell]}$$
(see equation (\ref{eq:psi+})) while $\psi_i^+(0) = q^h$. Thus, if $a > 0$ then we get (after simplifying and using (\ref{eq:rescale}))
$$q^c Q_i^{[1^{a+1}]} E_{i,b} - q^2 Q_i^{[1^a]} (E_{i,b+1}) = q^c E_{i,b} Q_i^{[1^{a+1}]} - q^{-2} E_{i,b+1} Q_i^{[1^a]}$$
while if $a=0$ we get
$$q^c (q-q^{-1}) Q_i E_{i,b} - q^2 E_{i,b+1} = q^c (q-q^{-1}) E_{i,b} Q_i - q^{-2} E_{i,b+1}.$$
Thus, we end up with
$$q^c [Q_i^{[1^{a+1}]}, E_{i,b}] =
\begin{cases}
q^2 Q_i^{[1^a]} (E_{i,b+1}) - q^{-2} (E_{i,b+1}) Q_i^{[1^a]} \text{ if } a > 0 \\
[2](E_{i,b+1}) \text{ if } a = 0.
\end{cases}$$
Now we will show that this is a consequence of the other relations. To simplify notation we will temporarily write $E_m$ for $E_{i,m}$ and $F_m$ for $F_{i,m}$ and $Q$ instead of $Q_i$.

{\bf Case $a=0$.} First we prove the case $a=0$ which says that $q^c [Q,E_b] = [2] E_{b+1}$. To do this we only use the two relations 
$$[E_{b+1}, F_{-b}] = q^{(b+1)c} q^h Q \text{ and } E_{b+1} E_b = q^2 E_b E_{b+1}.$$
We have
\begin{eqnarray*}
q^h q^c [Q,E_b] 
&=& q^c (q^h QE_b - q^2 E_b q^h Q) \\
&=& q^c q^{-(b+1)c} \big( E_{b+1}F_{-b}E_b - F_{-b}E_{b+1}E_b - q^2 E_bE_{b+1}F_{-b} + q^2 E_bF_{-b}E_{b+1} \big) \\
&=& q^{-bc} \big( E_{b+1}E_bF_{-b} - E_{b+1} [\ell+bc] - q^2F_{-b}E_bE_{b+1} - E_{b+1}E_bF_{-b} \\ 
& & + q^2F_{-b}E_bE_{b+1} + q^2 E_{b+1}[\ell+bc+2] \big) \\
&=& q^{-bc} \big( E_{b+1}(q^{\ell+bc+3} + q^{\ell+bc+1} \big) \\
&=& q^h [2] E_{b+1}
\end{eqnarray*}
where $\ell := \la \l, \alpha_i \ra$ and $\l$ is the weight space on the far right ({\it i.e.} the domain). Cancelling the $q^h$ completes the proof.

{\bf Case $a > 0$.} Here we need to show that 
\begin{equation}\label{eq:15}
q^{c} [Q^{[1^{a+1}]}, E_b] = q^2 Q^{[1^a]} E_{b+1} - q^{-2} E_{b+1} Q^{[1^a]}.
\end{equation}
We begin by computing $q^h$ times the left side of (\ref{eq:15}). We have 
\begin{eqnarray*}
q^h \cdot (LS) 
&=& q^{c}(q^h Q^{[1^{a+1}]} E_b - q^2 E_b q^h Q^{[1^{a+1}]}) \\
&=& q^{c} q^{-(a+b+1)c} \left( [E_{a+b+1}, F_{-b}] E_b - q^2 E_b [E_{a+b+1}, F_{-b}] \right) \\
&=& q^{-(a+b)c} \big( E_{a+b+1} E_b F_{-b} - E_{a+b+1} [\ell+bc] - F_{-b} E_{a+b+1} E_b \\
& & - q^2 E_b E_{a+b+1} F_{-b} + q^2 F_{-b} E_b E_{a+b+1} + q^2 E_{a+b+1} [\ell+2+bc] \big) \\
&=& q^{-(a+b)c} \big( (-E_{b+1} E_{a+b} + q^2 E_{a+b} E_{b+1}) F_{-b} + F_{-b}(E_{b+1} E_{a+b} - q^2 E_{a+b} E_{b+1}) \\ 
& & + q^{\ell+2+bc} [2] E_{a+b+1} \big)
\end{eqnarray*}
where $\ell$ is as above. Here we used that $[E_m, F_n] = q^{mc} q^h Q^{[1^{m+n}]}$ to obtain the second equality (where we take $m=a+b+1$ and $n=-b$), we use the standard relation for $[E_b,F_{-b}]$ to get the third equality, and we use the relation
\begin{equation}\label{eq:16}
E_m E_{n-1} + E_n E_{m-1} = q^2 ( E_{m-1} E_n + E_{n-1} E_m )
\end{equation}
to conclude the last equality.
\begin{Remark}
This relation appears as condition (8) in the vertex or idempotent realizations and is proved in the next subsection. The argument we employ is not circular because in that proof we only use the fact that $[P_i,E_{i,n}] = -[2](E_{i,n-1})$ which is the case $a=0$ proved above. To deal with this case we only use that $E_{b+1}E_b = q^2 E_bE_{b+1}$ which is one of the relations included in the definition in section \ref{sec:minimal}.  
\end{Remark}

Similarly, we compute $q^h$ times the right side of (\ref{eq:15}). We have 
\begin{eqnarray*}
q^h \cdot (RS)
&=& q^2 q^h Q^{[1^a]} E_{b+1} - E_{b+1} q^h Q^{[1^a]} \\
&=& q^{-(a+b)c} \big[ q^2(E_{a+b} F_{-b} - F_{-b} E_{a+b})E_{b+1} - E_{b+1}(E_{a+b} F_{-b} - F_{-b} E_{a+b}) \big] \\
&=& q^{-(a+b)c} \big[ q^2(E_{a+b}E_{b+1}F_{-b} - E_{a+b} q^{(b+1)c} q^h Q - F_{-b} E_{a+b} E_{b+1}) \\
& & (-E_{b+1} E_{a+b} F_{-b} + F_{-b} E_{b+1} E_{a+b} + q^{(b+1)c} q^h Q E_{a+b}) \big].
\end{eqnarray*}
Now, using the case $a=0$ relation proved above we get
$$- q^2 E_{a+b} q^h Q + q^h Q E_{a+b} = q^h [Q, E_{a+b}] = q^{-c} q^h [2] E_{a+b+1} = q^{\ell + 2 - c} [2] E_{a+b+1}.$$
Substituting we get that 
\begin{eqnarray*}
q^h \cdot (RS) 
&=& q^{-(a+b)c} \big[ (q^2 E_{a+b}E_{b+1} - E_{b+1} E_{a+b})F_{-b} + F_{-b}(E_{b+1}E_{a+b} - q^2 E_{a+b} E_{b+1}) \\
& & + q^{\ell+2+bc} [2] E_{a+b+1} \big] \\
&=& q^h \cdot (LS)
\end{eqnarray*}
and we are done. 

There are three other cases to prove, namely:
\begin{eqnarray*}
q^{-c/2} \psi_i^+(a+1) f_{i,b} - q^{-2} \psi_i^+(a) f_{i,b+1} &=& q^{-2} q^{-c/2} f_{i,b} \psi_i^+(a) - f_{i,b+1} \psi_i^+(a) \\
q^{-c/2} \psi_i^-(-a) e_{i,b} - q^2 \psi_i^-(-a-1) e_{i,b+1} &=& q^2 q^{-c/2} e_{i,b} \psi_i^-(-a) - e_{i,
b+1} \psi_i^+(-a-1) \\
q^{c/2} \psi_i^-(-a) f_{i,b} - q^{-2} \psi_i^-(-a-1) f_{i,b+1} &=& q^{-2} q^{c/2} f_{i,b} \psi_i^-(-a) - f
_{i,b+1} \psi_i^-(-a-1).
\end{eqnarray*}
These follow in precisely the same way as the proof above. 

\subsubsection{Case 2: $\la i, j \ra = -1$}
If we write out relation (6) when $s=+$ with $\psi_j^+$ and consider the coefficient of $[z^{-a}][w^{-b}]$ we get 
$$q^{c/2} \psi_j^+(a+1) e_{i,b} - q^{-1} \psi_j^+(a) e_{i,b+1} = q^{-1} q^{c/2} e_{i,b} \psi_j^+(a+1) - e_{i,b+1} \psi_j^+(a).$$
Substituting and simplifying leads to 
$$q^c [Q_j^{[1^{a+1}]}, E_{i,b}] =
\begin{cases}
-q(E_{i,b+1}) Q_j^{[1^a]} + q^{-1} Q_j^{[1^a]} (E_{i,b+1}) \text{ if } a > 0 \\
-(E_{i,b+1}) \text{ if } a = 0.
\end{cases}$$

{\bf Case $a=0$.} We need to show that $q^c [Q_j, E_{i,b}] = - E_{i,b+1}$. To do this we will use
\begin{eqnarray}
\label{eq:A} [E_{j,b+1}, F_{j,-b}] &=& q^{bc} q^{h_j} Q_j \\
\label{eq:B} E_{i,b+1} E_{j,b} + E_{j,b+1} E_{i,b} &=& q^{-1} E_{j,b} E_{i,b+1} + q^{-1} E_{i,b} E_{j,b+1}.
\end{eqnarray}
and that $E_i$'s and $F_j$'s commute. We have
\begin{eqnarray*}
q^{h_j} \cdot (LS)
&=& q^c q^{-bc-c} \big( E_{j,b+1} F_{j,-b} E_{i,b} - F_{j,-b} E_{j,b+1} E_{i,b} \\
& & -q^{-1} E_{i,b} E_{j,b+1} F_{j,-b} + q^{-1} E_{i,b} F_{j,-b} E_{j,b+1} \big) \\
&=& q^{-bc} \big( - E_{i,b+1} E_{j,b} F_{j,-b} + F_{j,-b} E_{i,b+1} E_{j,b} \\
& & +q^{-1} E_{j,b} E_{i,b+1} F_{j,-b} - q^{-1} F_{j,-b} E_{j,b} E_{i,b+1} \big) \\
&=& q^{-bc} \big( - E_{i,b+1} [\ell_j+bc] + q^{-1} E_{i,b+1} [\ell_j-1+bc] \big) \\
&=& -q^{-bc} E_{i,b+1} q^{\ell_j+bc-1} \\
&=& - q^{h_j} E_{i,b+1} \\
&=& q^{h_j} \cdot (RS)
\end{eqnarray*}
where $\ell_j := \la \l, j \ra$ and $\l$ is the weight space on the far right ({\it i.e.} the domain). Here we used (\ref{eq:A}) to get the first equality and (\ref{eq:B}) to get the second.

{\bf Case $a>0$.} This proof is similar to the one in case 1 so we omit it.

There are also three other cases to consider, namely:
\begin{eqnarray*}
q^{-c/2} \psi_j^+(a+1) f_{i,b} - q \psi_j^+(a) f_{i,b+1} &=& q q^{c/2} f_{i,b} \psi_j^+(a+1) - f_{i,b+1} \psi_j^+(a) \\
q^{-c/2} \psi_j^-(-a) e_{i,b} - q^{-1} \psi_j^-(-a-1) e_{i,b+1} &=& q^{-1} q^{-c/2} e_{i,b} \psi_j^-(-a) - e_{i,b+1} \psi_j^-(-a-1) \\
q^{c/2} \psi_j^-(-a) f_{i,b} - q \psi_j^-(-a-1) f_{i,b+1} &=& q q^{c/2} f_{i,b} \psi_j^-(-a) - f_{i,b+1} \psi_j^-(-a-1).
\end{eqnarray*}
These follow in the same way as the proof above.

\subsubsection{Case 3: $\la i, j \ra = 0$}
Relation (6) immediately simplifies to $\psi_j^s(z)x_i^{\pm}(w) = x_i^{\pm}(w)\psi_j^s(z)$. This implies that any $P_i^{[1^a]}$ or $Q_i^{[1^a]}$ commutes with any $E_{j,b}$ or $F_{j,b}$.

\subsection{Proof of (8)}

Writing out the condition gives:
\begin{equation}\label{eq:12}
(E_{i,m})(E_{i,n-1}) + (E_{i,n})(E_{i,m-1}) - q^2 \left[ (E_{i,m-1})(E_{i,n}) + (E_{i,n-1})(E_{i,m}) \right] = 0.
\end{equation}
Let us denote the left side of this equation by $f(m,n)$. Note that it is symmetric in that $f(m,n) = f(n,m)$. Now, from condition (3) we know that
$$[P_i, E_{i,n}] = - [2] (E_{i,n-1}).$$
Multiplying (\ref{eq:12}) on the left by $P_i$ and using this relation repeatedly we get
$$[P_i, f(m,n)] = - [2] \left[ f(m-1,n) + f(m,n-1) \right].$$
Thus $f(m,n)=0 \Rightarrow f(m-1,n) + f(m,n-1) = 0$. Applying $Q_i$ instead of $P_i$ likewise gives that $f(m,n)=0 \Rightarrow f(m+1,n) + f(m,n+1) = 0$.

Thus, if you know that $f(n,n) = 0$ then $f(n-1,n)=0$ since $f(n-1,n) = f(n,n-1)$ and then $f(n,n) + f(n-1,n+1) = 0$ which means $f(n-1,n+1)=0$. Continuing in this way one finds that
$$f(n,n)=0 \Rightarrow f(n+1,n)=0 \Rightarrow f(n+1,n-1)=0 \Rightarrow f(n+2,n-1)=0 \Rightarrow f(n+2,n-2)=0 \Rightarrow \dots$$
which means that $f(n,n)=0 \Rightarrow f(n+k,n-k)=0 \text{ and } f(n+k+1,n-k)=0$ for any $k \in \Z$. Thus it suffices to know that $f(n,n)=0$ for all $n \in \Z$ which is condition (4) from \ref{sec:vertex}.

\subsection{Proof of (9)}

Now suppose $\la i, j \ra = -1$. Writing out condition (9) from section \ref{sec:drinfeld} gives
$$(E_{i,m})(E_{j,n-1}) + (E_{j,n})(E_{i,m-1}) - q^{-1} \left[ (E_{j,n-1})(E_{i,m}) + (E_{i,m-1})(E_{j,n}) \right] = 0.$$
Again, let us denote the left hand side $f(m,n)$. This time we will use that
$$[P_i, E_{i,m}] = [2] (E_{i,m-1})
\text{ and }
[P_i, E_{j,n}] = - E_{j,n-1}.$$
Now, multiplying $f(m,n)$ by $P_i$ and using these relations gives
$$[P_i, f(m,n)] = [2] f(m-1,n) + q^{-1} f(m,n-1) $$
while multiplying by $P_j$ gives
$$[P_j, f(m,n)] = - q^{-1} f(m-1,n) + [2] f(m,n-1).$$
Putting this together gives
$$f(m,n)=0 \Rightarrow f(m-1,n)=0 \text{ and } f(m,n-1)=0.$$
Using $Q_i$ and $Q_j$ instead of $P_i$ and $P_j$ also gives us
$$f(m,n)=0 \Rightarrow f(m+1,n)=0 \text{ and } f(m,n+1)=0.$$
From this it follows that we only need to know $f(m,n)=0$ for one pair $(m,n)$. Taking $(m,n)=(1,1)$ is precisely condition (9) from section \ref{sec:vertex}.

Finally, if $\la i, j \ra = 0$ then writing out the condition in (9) from section \ref{sec:drinfeld} immediately simplifies to condition (9) in section \ref{sec:vertex}.

\subsection{Proof of (10)} 
The Drinfeld condition is equivalent to
\begin{equation}\label{eq:14}
\sum_{\sigma \in S_2} \left( E_{j,n} E_{i,m_{\sigma(1)}} E_{i,m_{\sigma(2)}} +  E_{i,m_{\sigma(1)}} E_{i,m_{\sigma(2)}} E_{j,n} \right)
= \sum_{\sigma \in S_2} [2] E_{i,m_{\sigma(1)}} E_{j,n} E_{i,m_{\sigma(2)}}
\end{equation}
if $\la i, j \ra = -1$. Notice that we have
\begin{eqnarray*}
& & q^{-1} (E_{i,m_1}) (E_{j,n}) (E_{i,m_2}) \\
&=& \left[ (E_{i,m_1+1})(E_{j,n-1}) + (E_{j,n})(E_{i,m_1}) - q^{-1}(E_{j,n-1})(E_{i,m_1+1}) \right] (E_{i,m_2})
\end{eqnarray*}
and likewise\begin{eqnarray*}
& & q (E_{i,m_1}) (E_{j,n}) (E_{i,m_2}) \\
&=& (E_{i,m_1})
\left[ (E_{j,n-1})(E_{i,m_2+1}) + (E_{i,m_2})(E_{j,n}) - q(E_{i,m_2+1})(E_{j,n-1}) \right].
\end{eqnarray*}
Adding and symmetrizing with respect to $m_1$ and $m_2$ we get the Drinfeld condition above but only if we can show that
\begin{eqnarray*}
& & \sum_{\sigma \in S_2}(E_{i,m_{\sigma(1)+1}})(E_{j,n-1})(E_{i,m_{\sigma(2)}}) +(E_{i,m_{\sigma(1)}})(E_{j,n-1})(E_{i,m_{\sigma(2)}+1}) \\
&=&\sum_{\sigma \in S_2} q^{-1}(E_{j,n-1})(E_{i,m_{\sigma(1)}+1})(E_{i,m_{\sigma(2)}}) + q(E_{i,m_{\sigma(1)}})(E_{i,m_{\sigma(2)}+1})(E_{j,n-1}).
\end{eqnarray*}

Now multiply both sides of this equation by $q+q^{-1}$. If we collect the terms on the right side with a coefficient of $q^2$ or $q^{-2}$ and use the appropriate commutator relation (\ref{eq:12}) (appropriate means that we should get rid of the $q$ coefficients) then we find that the condition above is equivalent to condition (\ref{eq:14}) for the values $(n-1,m_1,m_2+1)$ and $(n-1,m_1+1,m_2)$.

Equivalently, this means that if $m_1 < m_2$ then (\ref{eq:14}) for $(n,m_1,m_2)$ is implied by $(n+1,m_1-1,m_2)$ and $(n,m_1-1,m_2+1)$. Thus repeating the argument we can reduce to the case when $m_1=m_2=m$. In this case (\ref{eq:14}) becomes
\begin{equation*}
(E_{j,n})(E_{i,m})^2 + (E_{i,m})^2 (E_{j,n}) = [2](E_{i,m})(E_{j,n})(E_{i,m}).
\end{equation*}
If $m=n=0$ then this relation follows formally from the other relations (see for instance \cite{Nak2}). But in the more general case, $E_i' := (E_{i,m})$ and $E_j' := (E_{j,n})$ also generate an $\sl_3$ action (when combined with $F_i' := (F_{i,-m})$ and $F_j' := (F_{j,-n})$) so the argument used in the case $m=n=0$ still applies.

\section{Minimal realization}\label{sec:conj}

The presentation from Section \ref{sec:minimal} can probably be stripped down even further. For generators, consider the algebra generated by $E_{i,r} \1_\l$ and $F_{i,r} \1_\l$, where $\1_\l$ are  idempotents as before and where $i \in \I \text{ and } r \in \Z$. The relations are as follows. 

\begin{enumerate}
\item $\{\1_\l: \l \in \hX \}$ are mutually orthogonal idempotents with
\begin{eqnarray*}
& & E_{i,r} \1_\l = \1_\mu E_{i,r} \1_\l = \1_\mu E_{i,r} \\
& & F_{i,-r} \1_\mu = \1_\l F_{i,-r} \1_\mu = \1_\l F_{i,-r}
\end{eqnarray*}
where $\mu = \l+ \alpha_i + rc \delta$
\item For $a,b,a',b' \in \Z$ we have 
\begin{eqnarray*}
q^{-ac} [E_{i,a}, F_{i,b}] \1_\l &=& q^{-a'c} [E_{i,a'}, F_{i,b'}] \1_\l \text{ if } a+b=a'+b' > 0 \cr
q^{-bc} [E_{i,a}, F_{i,b}] \1_\l &=& q^{-b'c} [E_{i,a'}, F_{i,b'}] \1_\l \text{ if } a+b=a'+b' < 0 \cr
[E_{i,a}, F_{i,b}] \1_\l &=& [\la \l, i \ra + ac] \1_\l \text{ if } a+b=0.
\end{eqnarray*}
If $i \ne j$ then $[E_{i,a}, F_{j,b}] \1_\l = 0$.
\item For any $n \in \Z$ we have $E_{i,n}E_{i,n-1} = q^2 E_{i,n-1}E_{i,n}$ and $F_{i,n-1}F_{i,n} = q^2 F_{i,n}F_{i,n-1}$.
\item For any $i \ne j \in \I$ and $n \in \Z$ we have
\begin{eqnarray*}
E_{i,1}E_j \1_\l + E_{j,1}E_i \1_\l = q^{-1} \left( E_jE_{i,1} \1_\l + E_iE_{j,1} \1_\l \right) &\text{ if }& \la i, j \ra = -1 \\
E_i E_{j,n} \1_\l = E_{j,n}E_i \1_\l &\text{ if }& \la i, j \ra = 0
\end{eqnarray*}
and similarly
\begin{eqnarray*}
F_{i,-1}F_j \1_\l + F_{j,-1}F_i \1_\l = q^{-1} \left( F_jF_{i,-1} \1_\l + F_iF_{j,-1} \1_\l \right) &\text{ if }& \la i, j \ra
 = -1 \\
F_i F_{j,n} \1_\l = F_{j,n} F_i \1_\l &\text{ if }& \la i, j \ra = 0.
\end{eqnarray*}
\end{enumerate}

\begin{conj}\label{conj:1} In an integrable representation, all the Drinfeld relations are a consequence of the relations above. 
\end{conj}

The most glaring omission above are the $P$s and $Q$s. These can be defined using the $E$s and $F$s as follows
$$P_i^{[1^{-a-b}]} \1_\l := q^{-ac} q^{-\la \l,i \ra} [E_{i,a}, F_{i,b}] \1_\l \text{ and } Q_i^{[1^{a+b}]} \1_\l := q^{-bc} q^{\la \l,i \ra} [E_{i,a}, F_{i,b}] \1_\l.$$

\begin{lemma}\label{lem:2}
Let $P_i$ and $Q_i$ be defined by $-q^{-1}P_i=P_i^{[1]}$ and $-qQ_i=Q_i^{[1]}$.  The relations between $E_{i,a}$ and $F_{i,b}$ imply that $[Q_i,P_i] \1_\l = [2][c] \1_\l$, $[Q_i,P_j] \1_\l =-[c] \1_\l$ if $\la i, j \ra = -1$, and $[Q_i,P_i] \1_\l = 0$ otherwise. 
\end{lemma}
\begin{proof}
First we compute $Q_iP_i \1_\l$. To do this we use
$$P_i \1_\l = q^c q^{-\la \l, i \ra} (E_{i,-1}F_i \1_\l - F_i,E_{i,-1}) \text{ and } 
\1_\l Q_i = q^{-c} q^{\la \l,i \ra}(E_i F_{i,1} \1_\l - F_{i,1} E_i \1_\l).$$
Thus 
$$Q_iP_i \1_\l = E_iF_{i,1}E_{i,-1}F_i \1_\l - E_iF_{i,1}F_iE_{i,-1}\1_\l - F_{i,1}E_iE_{i,-1}F_i\1_\l + F_{i,1}E_iF_iE_{i,-1}\1_\l.$$
Consider the first term. We have 
\begin{align*}
& E_iF_{i,1}E_{i,-1}F_i \1_\l \\
=& E_iE_{i,-1}F_{i,1}F_i \1_\l - [\la \l-i,i \ra - c]E_iF_i \1_\l \\
=& E_{i,-1}E_iF_iF_{i,1} \1_\l - [\la \l,i \ra - 2 - c]E_iF_i \1_\l \\
=& E_{i,-1}F_iE_iF_{i,1} \1_\l + [\la \l-i,i \ra ]E_{i,-1}F_{i,1} \1_\l - [\la \l,i \ra -2-c]F_iE_i \1_\l - [\la \l,i \ra - 2 - c][\la \l,i \ra] \1_\l.
\end{align*}
Likewise, the other three terms give
\begin{equation*}\begin{split}
& -q^{-2} \left( F_iE_iE_{i,-1}F_{i,1} \1_\l - F_iE_i\1_\l [\la \l,i \ra - c] + E_{i,-1}F_{i,1}\1_\l [\la \l,i \ra] - [\la \l,i \ra ][\la \l,i \ra - c]\1_\l \right) \\
& -q^2 \left(E_{i,-1}F_{i,1}F_iE_i\1_\l + E_{i-1}F_{i,1}\1_\l[\la \l,i \ra] - [\la \l,i \ra - c]F_iE_i\1_\l - [\la \l,i \ra -c][\la \l,i \ra]\1_\l \right) \\
& F_iE_{i,-1}F_{i,1}E_i \1_\l - F_iE_i \1_\l [\la \l,i \ra +2-c] + [\la \l,i \ra + 2]E_{i,-1}F_{i,1} \1_\l - [\la \l,i \ra+2][\la \l,i \ra - c]\1_\l
\end{split}\end{equation*}
Now consider the coefficient of $F_iE_i \1_\l$ after summing up these four expressions. It equals 
$$-[\la \l,i \ra -2-c] + q^{-2}[\la \l,i \ra-c] + q^2[\la \l,i \ra-c] - [\la \l,i \ra +2-c]=0.$$
Likewise the coefficient of $E_{i,-1}F_{i,1}$ is zero. The coefficient of $\1_\l$ is 
$$-[\la \l,i \ra -2-c][\la \l,i \ra] + q^{-2}[\la \l,i \ra][\la \l,i \ra -c] + q^2[\la \l,i \ra][\la \l,i \ra -c] - [\la \l,i \ra+2][\la \l,i \ra-c]$$
which simplifies to give $[2][c]$. The remaining terms are grouped together to give $P_iQ_i \1_\l$. Thus we end up with $Q_iP_i \1_\l = P_iQ_i \1_\l + [2][c] \1_\l$. 

An analogous but slightly longer calculation, which we omit, can be used to show that $Q_iP_j \1_\l = P_j Q_i \1_\l - [c] \1_\l$ when $\la i,j \ra = -1$.  The fact that $Q_iP_j \1_\l = P_j Q_i \1_\l$ when $\la i,j \ra= 0$ is immediate.

\end{proof}

In general, the commutator relations beween $P_i^{[1^n]}$ and $Q_j^{[1^n]}$ can be read off from the Drinfeld realization. They are defined recursively as follows.
\begin{itemize}
\item If $i=j$ and $a,b \ge 0$ then
\begin{eqnarray*}
&& P_i^{[1^a]} Q_j^{[1^b]} - (q q^{-c} + q^{-1} q^c) P_i^{[1^{a+1}]} Q_j^{[1^{b+1}]} + P_i^{[1^{a+2}]} Q_j^{[1^{b+2}]} \\
&=& Q_j^{[1^b]} P_i^{[1^a]} - (q^{-1}q^{-c} + q q^c) Q_j^{[1^{b+1}]} P_i^{[1^{a+1}]} + Q_j^{[1^{b+2}]} P_i^{[1^{a+2}]}
\end{eqnarray*}
while $[Q_j^{[1^{b+1}]}, P_i^{[1]}] = -(q-q^{-1})[c] Q_j^{[1^b]}$ and $[Q_j^{[1]}, P_i^{[1^{b+1}]}] = (q-q^{-1})[c] P_i^{[1^b]}$ unless $b=0$ in which case $[Q_j^{[1]},P_i^{[1]}] = -[c]$.
\item If $\la i, j \ra = -1$ and $a,b \ge 0$ then
\begin{eqnarray*}
&& P_i^{[1^a]} Q_j^{[1^b]} - (q q^{-c} + q^{-1} q^c) P_i^{[1^{a+1}]} Q_j^{[1^{b+1}]} + P_i^{[1^{a+2}]} Q_j^{[1^{b+2}]} \\
&=& Q_j^{[1^b]} P_i^{[1^a]} - (q^{-1}q^{-c} + q q^c) Q_j^{[1^{b+1}]} P_i^{[1^{a+1}]} + Q_j^{[1^{b+2}]} P_i^{[1^{a+2}]}
\end{eqnarray*}
while $[Q_j^{[1^{b+1}]}, P_i^{[1]}] = -(q-q^{-1})[c] Q_j^{[1^b]}$ and $[Q_j^{[1]}, P_i^{[1^{b+1}]}] = (q-q^{-1})[c] P_i^{[1^b]}$ unless $b=0$ in which case $[Q_j^{[1]},P_i^{[1]}] = -[c]$.
\item If $\la i, j \ra = 0$ then $P_i^{[1^a]} Q_j^{[1^b]} = Q_j^{[1^b]} P_i^{[1^a]}$. 
\end{itemize}
One would expect that the proof in Lemma \ref{lem:2} extends to deduce these relations as a formal consequence of the commutator relations between $E$s and $F$s. However, such a computation is much more complex than the one in Lemma \ref{lem:2} and we did not perform it. 

\section{Remarks}\label{sec:remarks}

\subsection{Break of symmetry}

The algebra generated by $P$s and $Q$s is known as the quantum Heisenberg algebra. It has a natural involution $\psi$ defined by 
$$P_i^{(1^n)} \mapsto (-1)^n P_i^{(n)} \text{ and } Q_i^{(1^n)} \mapsto (-1)^n Q_i^{(n)}.$$
In terms of the generators $h_{i,m}$ this involution corresponds to $h_{i,m} \mapsto - h_{i,m}$. In section \ref{sec:vertex} we used these Heisenberg generators to define two sets of algebra elements, $\{P_i^{[1^n]}, Q_i^{[1^n]}\}$ and $\{P_i^{[n]}, Q_i^{[n]}\}$.  However, in subsequent presentations we used only $\{P_i^{[1^n]}, Q_i^{[1^n]}\}$ and never used $\{P_i^{[n]},Q_i^{[n]}\}$. This apparent break in symmetry is explained by the definition 
$$\psi_i^\pm(z) = \sum_{n \geq 0} \psi_i^\pm(\pm n) z^{\mp n} = q^{\pm h_i} \exp \big(\pm (q-q^{-1}) \sum_{n>0} h_{i, \pm n} z^{\mp n} \big)$$
from section \ref{sec:drinfeld}. Changing $h_{i,m}$ to $-h_{i,m}$ in this definition would have the effect of exchanging $P_i^{[1^n]} \leftrightarrow P_i^{[n]}$ and $Q_i^{[1^n]} \leftrightarrow Q_i^{[n]}$, because $\psi(P_i^{[1^n]}) = (-1)^n P_i^{[n]}$ and $\psi(Q_i^{[1^n]}) = (-1)^n Q_i^{[n]}$. 

\subsection{Divided powers}

To define the quantum affine algebra over $\Z[q,q^{-1}]$ one needs to include divided powers, as observed by Lusztig. These divided powers are defined as follows
\begin{equation}\label{eq:divpowers}
E_{i,m}^{(r)} := \frac{E_{i,m}^r}{[r]!} \text{ and } F_{i,n}^{(r)} := \frac{F_{i,n}^r}{[r]!}
\end{equation}
where $r \ge 0$ (by convention, $r=0$ gives the identity). Ideally one would like to understand the relations involving these divided powers; these relations should be defined over $\Z[q,q^{-1}]$, and they should arise in geometric and categorical constructions. The lemma below gives some examples of relations involving divided powers and $P$s and $Q$s. 

\begin{lemma}\label{lem:example}
The following identities hold:
\begin{eqnarray*}
[Q_i, E_{i,b}^{(r)}] = q^{r-1} [2] E_{i,b}^{(r-1)} E_{i,b+1} \ \ && \ \ [Q_i, F_
{i,b}^{(r)}] = - q^c q^{r-1} [2] F_{i,b+1} F_{i,b}^{(r-1)} \cr
[P_i, E_{i,b+1}^{(r)}] = q^{-c} q^{r-1} [2] E_{i,b} E_{i,b+1}^{(r-1)} \ \ && \ \
 [P_i, F_{i,b+1}^{(r)}] = - q^{r-1} [2] F_{i,b+1}^{(r-1)} F_{i,b} \cr
[Q_j, E_{i,b}^{(r)}] = - q^{r-1} E_{i,b}^{(r-1)} E_{i,b+1} \ \ && \ \ [Q_j, F_{i,b}^{(r)}] = q^c q^{r-1} F_{i,b+1} F_{i,b}^{(r-1)} \cr
[P_j, E_{i,b+1}^{(r)}] = - q^{-c} q^{r-1} E_{i,b} E_{i,b+1}^{(r-1)} \ \ && \ \ [P_i, F_{i,b+1}^{(r)}] = q^{r-1} F_{i,b+1}^{(r-1)} F_{i,b} 
\end{eqnarray*}
\end{lemma}
\begin{proof}
We prove the first relation by induction on $r$ (the others are proved in exactly the same way). The base case $r=1$ is part of the definition. To prove the induction step we have\begin{eqnarray*}Q_i E_{i,b}^{(r)} E_{i,b}
&=& E_{i,b}^{(r)} Q_i E_{i,b} + q^{r-1}[2] E_{i,b}^{(r-1)} E_{i,b+1} E_{i,b} \\
&=& E_{i,b}^{(r)} E_{i,b} Q_i + [2] E_{i,b}^{(r)} E_{i,b+1} + q^{r-1}[2] q^2 E_{i,b}^{(r-1)} E_{i,b} E_{i,b+1} \\
&=& E_{i,b}^{(r)} E_{i,b} Q_i + E_{i,b}^{(r)} E_{i,b+1} [2] \left( 1 + q^{r-1} q^2 [r] \right) \\
&=& E_{i,b}^{(r)} E_{i,b} Q_i + E_{i,b}^{(r)} E_{i,b+1} [2] [r+1] q^r.
\end{eqnarray*}
This means that $[r+1] [Q_i, E_{i,b}^{(r+1)}] = q^r [2] [r+1] E_{i,b}^{(r)} E_{i,b+1}$ which implies the induction step.
\end{proof}

Let $U_q^\Z(\hg)$ be the subalgebra of $U_q(\hg)$ generated over $\Z[q,q^{-1}]$ by the divided powers $E_{i,m}^{(r)}$ and $F_{i,n}^{(r)}$ together with $Q^{[1^a]}$ and $P^{[1^a]}$. As was noted to us by the referee, it follows from Theorem 2 of \cite{BCP} that the natural map $U_q^\Z(\hg) \otimes_{\Z[q,q^{-1}]} \k(q) \rightarrow U_q(\hg)$ is an isomorphism. It is natural to conjecture that $U_q^\Z(\hg)$ is also free over $\Z[q,q^{-1}]$, although we have not proven this (this observation appears in \cite{CP1} and also in the remark in section 1 of \cite{Nak1}).

\subsection{Renormalization}\label{sec:renormalized}

In \cite{CL2} we construct a categorical action of the quantum affine algebra using the Heisenberg algebra from \cite{CL1}. However, the resulting decategorified relations do not match up identically with those in the above presentations. More precisely, they are off by a sign or some power of $q$ in a few instances. We now renormalize the generators of the idempotent vertex realization $\dU_q(\hg)$ of the quantum affine algebra so that the resulting presentation matches up with that in \cite{CL2}. So the point of this renormalized realization of $\dU_q(\hg)$ is that it occurs naturally in categorification.

To define the renormalization we need an asymmetric pairing $(\cdot,\cdot)$ on $\hX$. To define this pairing, fix an orientation of the original Dynkin diagram of $\g$, so that an edge between $i$ and $j$ in the Dynkin diagram is now oriented $i \rightarrow j$ or $j\rightarrow i$, but not both. Then 
$$(i,j) := 
\begin{cases}
1 \text{ if } i = j \\
-1 \text{ if } i \rightarrow j \\
0 \text{ otherwise }
\end{cases}$$
while $(i,\delta) = 0 = (\delta, i)$ and $(\delta,\delta) = 0$. Moreover, $(\Lambda_i,j) = \delta_{i,j}$ and $(j,\Lambda_i)=0$ while $(\Lambda_i,\delta) = 1$ and $(\delta,\Lambda_i)=0$. In particular, this means that if $\l$ is a weight appearing in the representation of level $c$ then $(\l,\delta) = c$ and $(\delta,\l) = 0$. It is easy to check using this definition that $\la \l_1,\l_2 \ra = (\l_1,\l_2) + (\l_2,\l_1)$ for any $\l_1 \in \hX$ and $\l_2 \in \hY$. 

We define the renormalization as follows. It takes $q \mapsto -q$ and 
$$E_{i,m} \1_\l \mapsto (-1)^{(\l,i)} q^{mc} E_{i,m} \1_\l \ \ \ \ \1_\l F_{i,m} \mapsto - (-1)^{(i,\l)} (-q)^{mc} \1_\l F_{i,m}$$
$$P_i^{(n)} \1_\l \mapsto - P_i^{(n)} \1_\l \ \ P_i^{(1^n)} \1_\l \mapsto - P_i^{(1^n)} \1_\l \ \ 
\1_\l Q_i^{(n)} \mapsto - (-1)^{cn} \1_\l Q_i^{(n)} \ \ \1_\l Q_i^{(1^n)} \mapsto - (-1)^{cn} \1_\l Q_i^{(1^n)}$$
where $c = (\l,\delta)$ is the level. 

For example, the relation $q^c [Q_i,E_{i,b}] \1_\l = [2] E_{i,b+1} \1_\l$ becomes 
$$(-q)^c \left( (-1)^{c+1} Q_i (-1)^{(\l,i)} q^{bc} E_{i,b} - (-1)^{(\l+\delta,i)} q^{bc} E_{i,b} (-1)^{c+1} Q_i \right) \1_\l = -[2] (-1)^{(\l,i)} q^{(b+1)c} E_{i,b+1} \1_\l$$
which simplifies to give $[Q_i,E_{i,b}] \1_\l = [2] E_{i,b+1} \1_\l$. The resulting relations turn out to eb independent of the choice of orientation of the Dynkin diagram. We summarize them below. 

The renormalized (idempotent) vertex realization has generators 
$$E_{i,r} \1_\l, F_{i,r} \1_\l, Q_i^{(n)} \1_\l, P_i^{(n)} \1_\l, Q_i^{(1^n)} \1_\l, P_i^{(1^n)} \1_\l, \text{ where } i \in \I \text{ and } r,k \in \Z.$$
The $P_i^{[1^n]},Q_i^{[1^n]},P_i^{[n]}, Q_i^{[n]}$ are defined as before. The relations in this renormalized realization are then modified as follows.

\begin{enumerate}
\item This condition is redundant.
\item This condition is unchanged. 
\item This condition is unchanged. 
\item We have 
\begin{eqnarray}
\label{eq:PQ1'} Q_j^{(n)} P_i^{(m)} \1_\l &=&
\begin{cases}
\sum_{k \ge 0} \Sym^k([2][c]) P_i^{(m-k)} Q_i^{(n-k)} \1_\l \text{ if } i = j \\
\sum_{k \ge 0} \Lambda^k([c]) P_i^{(m-k)} Q_j^{(n-k)} \1_\l \text{ if } \la i, j \ra = -1 \\
P_i^{(m)} Q_j^{(n)} \1_\l \text{ if } \la i, j \ra = 0
\end{cases} \\
\label{eq:PQ2'}
Q_j^{(1^n)} P_i^{(m)} \1_\l &=&
\begin{cases}
\sum_{k \ge 0} \Lambda^k([2][c]) P_i^{(m-k)} Q_i^{(1^{n-k})} \1_\l \text{ if } i = j \\
\sum_{k \ge 0} \Sym^k([c]) P_i^{(m-k)} Q_j^{(1^{n-k})} \1_\l \text{ if } \la i, j \ra = -1 \\
P_i^{(m)} Q_j^{(1^n)} \1_\l \text{ if } \la i, j \ra = 0
\end{cases}
\end{eqnarray}
and likewise if you exchange $(a)$ with $(1^a)$ everywhere.

\item We have
\begin{eqnarray*}
P_i^{(n)} \1_\l = \1_\mu P_i^{(n)} \1_\l = \1_\mu P_i^{(n)} &\text{ and }& P_i^{(1^n)} \1_\l = \1_\mu P_i^{(1^n)} \1_\l = \1_\mu P_i^{(1^n)} \\
Q_i^{(n)} \1_\mu = \1_\l Q_i^{(n)} \1_\mu = \1_\l Q_i^{(n)} &\text{ and }& Q_i^{(1^n)} \1_\mu = \1_\l Q_i^{(1^n)} \1_\mu = \1_\l Q_i^{(1^n)}
\end{eqnarray*}
where $\mu = \l + nc \delta$.

\item We have 
\begin{eqnarray*}
\label{eq:q'E1}
[Q_i^{[1^{a+1}]}, E_{i,b}] \1_\l &=&
\begin{cases}
q^2 Q_i^{[1^a]} E_{i,b+1} \1_\l - q^{-2} E_{i,b+1} Q_i^{[1^a]} \1_\l \text{ if } a > 0 \\
[2]E_{i,b+1} \1_\l \text{ if } a = 0.
\end{cases} \\
\label{eq:q'E2}
q^{-c} [Q_i^{[1^{a+1}]}, F_{i,b}] \1_\l &=&
\begin{cases}
q^{-2} Q_i^{[1^a]} F_{i,b+1} \1_\l - q^2 F_{i,b+1} Q_i^{[1^a]}  \1_\l \text{ if } a > 0 \\
- [2]F_{i,b+1}\1_\l \text{ if } a = 0.
\end{cases} \\
\label{eq:q'E3}
q^c [P_i^{[1^{a+1}]}, E_{i,b+1}]\1_\l &=&
\begin{cases}
q^2 E_{i,b} P_i^{[1^{a}]} \1_\l - q^{-2} P_i^{[1^{a}]} E_{i,b}\1_\l \text{ if } a > 0 \\
[2] E_{i,b}\1_\l \text{ if } a = 0
\end{cases} \\
\label{eq:q'E4}
[P_i^{[1^{a+1}]}, F_{i,b+1}]\1_\l &=&
\begin{cases}
q^{-2} F_{i,b} P_i^{[1^{a}]} \1_\l - q^2 P_i^{[1^{a}]} F_{i,b}\1_\l \text{ if } a > 0 \\
- [2] F_{i,b}\1_\l \text{ if } a = 0.
\end{cases}
\end{eqnarray*}
while if $\la i, j \ra = -1$ we have
\begin{eqnarray*}
\label{eq:q'E5}
[Q_j^{[1^{a+1}]}, E_{i,b}] \1_\l &=&
\begin{cases}
qE_{i,b+1} Q_j^{[1^a]}\1_\l - q^{-1} Q_j^{[1^a]} E_{i,b+1}\1_\l \text{ if } a > 0 \\
E_{i,b+1}\1_\l \text{ if } a = 0.
\end{cases} \\
\label{eq:q'E6}
q^{-c} [Q_j^{[1^{a+1}]}, F_{i,b}]\1_\l &=&
\begin{cases}
q^{-1} F_{i,b+1} Q_j^{[1^a]}\1_\l - q Q_j^{[1^a]} F_{i,b+1}\1_\l \text{ if } a > 0 \\
- F_{i,b+1}\1_\l \text{ if } a = 0
\end{cases} \\
\label{eq:q'E7}
q^c [P_j^{[1^{a+1}]}, E_{i,b+1}]\1_\l &=&
\begin{cases}
q^{-1}E_{i,b} P_j^{[1^{a}]}\1_\l - q P_j^{[1^{a}]} E_{i,b}\1_\l \text{ if } a > 0 \\
E_{i,b}\1_\l \text{ if } a = 0
\end{cases} \\
\label{eq:q'E8}
[P_j^{[1^{a+1}]}, F_{i,b+1}]\1_\l &=&
\begin{cases}
qF_{i,b} P_j^{[1^{a}]}\1_\l - q^{-1} P_j^{[1^{a}]} F_{i,b}\1_\l \text{ if } a > 0 \\
- F_{i,b}\1_\l \text{ if } a = 0.
\end{cases}
\end{eqnarray*}

\item We have
$$[E_{i,a}, F_{i,b}] \1_\l =
\begin{cases}
q^{-bc} q^{\la \l, i \ra} Q_i^{[1^{a+b}]} \1_\l \text{ if } a+b > 0 \\
q^{-ac} q^{- \la \l, i \ra} P_i^{[1^{-a-b}]} \1_\l \text{ if } a+b < 0 \\
[\la \l,i \ra + ac] \1_\l \text{ if } a+b=0.

\end{cases}$$
while if $i \ne j$ then $[E_{i,a}, F_{j,b}] \1_\l = 0$.

\item This condition is unchanged. 

\item For any $m,n \in \Z$, if $\la i, j \ra = -1$ we have
\begin{eqnarray*}
E_{i,m}E_{j,n+1} \1_\l - q E_{j,n+1} E_{i,m} \1_\l &=& E_{j,n} E_{i,m+1} \1_\l - q E_{i,m+1} E_{j,n} \1_\l \\
F_{i,m+1}F_{j,n} \1_\l - q F_{j,n} F_{i,m+1} \1_\l &=& F_{j,n+1} F_{i,m} \1_\l - q F_{i,m} F_{j,n+1} \1_\l 
\end{eqnarray*}
while if $\la i, j \ra = 0$ then
$$E_{i,m}E_{j,n} \1_\l = E_{j,n}E_{i,m} \1_\l \text{ and } F_{i,m}F_{j,n} \1_\l = F_{j,n}F_{i,m} \1_\l.$$

\item This condition is unchanged. 
\end{enumerate}

\end{document}